\pdfoutput=1
\documentclass[12pt]{amsart}

\title[Random walks can not be slowed]{Random walks on regular trees can not be slowed down}

\author{Omer Angel}
\address{Department of Mathematics, University of British Columbia}
\email{angel@math.ubc.ca}

\author{Jacob Richey}
\address{Department of Mathematics, University of British Columbia}
\email{jfrichey@math.ubc.ca}
\thanks{}

\author{Yinon Spinka}
\address{Department of Mathematics, University of British Columbia\hfill\break\indent School of Mathematical Sciences. Tel Aviv University. Tel Aviv 6997801, Israel}
\email{yinon@math.ubc.ca}
\thanks{}

\author{Amir Yehudayoff}
\address{Department of Mathematics, Technion-IIT}
\email{amir.yehudayoff@gmail.com}
\thanks{}

\usepackage{amsmath,amssymb,amsfonts,amsthm}
\usepackage[protrusion=true,expansion=true]{microtype}
\usepackage{color}

\usepackage[colorlinks=true,linkcolor=blue,citecolor=blue,urlcolor=blue,pdfborder={0 0 0}]{hyperref}
\usepackage[nameinlink]{cleveref}
\usepackage{bbm}
\usepackage{enumitem}
\usepackage{cite}
\usepackage{nicefrac,setspace}
\usepackage{xcolor}

\usepackage[margin=3cm]{geometry}

\crefname{theorem}{Theorem}{Theorems}
\crefname{lemma}{Lemma}{Lemmas}
\crefname{remark}{Remark}{Remarks}
\crefname{prop}{Proposition}{Propositions}
\crefname{defn}{Definition}{Definitions}
\crefname{corollary}{Corollary}{Corollaries}
\crefname{conjecture}{Conjecture}{Conjectures}
\crefname{question}{Question}{Questions}
\crefname{chapter}{Chapter}{Chapters}
\crefname{section}{Section}{Sections}
\crefname{figure}{Figure}{Figures}
\crefname{obs}{Observation}{Observations}

\theoremstyle{plain}
\newtheorem{theorem}{Theorem}[section]
\newtheorem{lemma}[theorem]{Lemma}

\newtheorem{prop}[theorem]{Proposition}

\theoremstyle{definition}

\theoremstyle{remark}
\newtheorem{remark}[theorem]{Remark}
\newtheorem{obs}[theorem]{Observation}

\numberwithin{equation}{section}

\newcommand{\eps}{\varepsilon}

\newcommand{\N}{\mathbb N}
\newcommand{\R}{\mathbb R}
\newcommand{\E}{\mathbb E}
\newcommand{\Z}{\mathbb Z}
\renewcommand{\P}{\mathbb P}
\renewcommand{\Pr}{\mathbb P}

\newcommand{\T}{\mathbb T}
\newcommand{\iso}{\mathsf{iso}}
\newcommand{\conn}{\mathsf{con}}

\usepackage[medium]{titlesec}
\titleformat{\section}{\Large\normalfont\bfseries}{\thesection}{1em}{}
\titleformat{\subsection}{\normalfont\bfseries}{\thesubsection}{1em}{}
\titleformat{\subsubsection}{\normalfont\bfseries}{\thesubsubsection}{1em}{}
\titleformat{\paragraph}[runin]{\bfseries}{\theparagraph}{}{}
\titlespacing*{\section}{0pt}{3.5ex plus 1ex minus .2ex}{2.3ex plus .2ex}
\titlespacing*{\subsection}{0pt}{3.25ex plus 1ex minus .2ex}{1.5ex plus .2ex}
\titlespacing*{\subsubsection}{0pt}{3.25ex plus 1ex minus .2ex}{1.5ex plus .2ex}
\titlespacing*{\paragraph}{0pt}{1ex plus .2ex minus .2ex}{1.5ex plus .2ex}

\begin{document}

\begin{abstract}
  A random walk on a regular tree (or any non-amenable graph) has positive speed.
  We ask whether such a walk can be slowed down by applying carefully chosen time-dependent permutations of the vertices.
  We prove that on trees the random walk can not be slowed down.
\end{abstract}

\maketitle

\section{Introduction}

One of the classical results relating the geometry of a space to the behaviour of random walks on the space is that on any non-amenable graph the random walk has positive speed, in that $\liminf_{t \to\infty} t^{-1}|X_t|$ exists and is a.s.\ positive.
On transitive graphs the limit is even an almost sure constant.
In particular, on the $d$ regular tree, denoted $\T_d$, the speed for the simple random walk is $\frac{d-2}{d}$, which is positive as long as $d>2$.
The motivation for this paper is the question: {\em Can we slow down the particle?}

Suppose that after each step $t$ of the random walk, we are allowed to apply some permutation $\pi_t$ to the vertices of the tree, so that if the particle is at $v$ it is transported to $\pi_t(v)$.
If we observe the particle and can choose $\pi_t$ accordingly, then we can constantly push it back to any vertex we wish, so that it never moves.
Our main finding is that if the permutations do not depend on the location of the particle, then the particle can not be slowed down.

\subsection{Permuted random walks}

We start by considering lazy random walks, where the results are cleaner for mostly technical reasons (see the discussion below).
We start by introducing some notations.
Fix $d \geq 2$, and let $\T = \T_d$ denote the rooted infinite $d$-regular tree.
The vertex set is denoted by $V = V(\T_d)$. 
The root of the tree is denoted by $v_0$.
The depth $|v|$ of a vertex $v \in V$ is its distance from the root.
The neighborhood $N(v)$ of $v$ is the set of vertices $u$ that are of distance at most one from $v$.
Note that since we are considering lazy random walks, it is convenient to have $v\in N(v)$.
Thus the size of $N(v)$ is $d+1$.

Let $(X_t)_{t=0}^\infty$ be a lazy random walk on $\T$ started at the root.
The laziness parameter $\P(X_{t+1}=X_t)$ is chosen to be $\frac1{d+1}$. 
That is, $X_0 = v_0$ and $X_{t+1}$ is a uniformly random element of $N(X_t)$.
The (empirical) speed of $(X_t)$ is defined to be the process $(t^{-1} |X_t|)$.
The strong law of large numbers implies that the speed a.s.\ converges to $\frac{d-2}{d+1}$.
Note that this also holds in the case $d=2$ where $\T_2$ is the line and the speed is $0$.
For more on random walks on trees see e.g.~\cite{lyons2017probability,virag2000anchored} and references therein. 

The model we suggest for studying the slowing down of particles is as follows.
Before the particle starts to move, we can choose a sequence $(\pi_t)_{t=1}^\infty$ of permutations of $V$.
(These do not need to be finitary; any bijections of $V$ will do.)
The permutation $\pi_t$ is applied on the random walk at time $t$.
Thus the permuted random walk $(Y_t)$ starts at the root, and its position at time $t+1$ is defined by $Y_{t+1} = \pi_{t+1}(Y'_{t+1})$,
where $Y'_{t+1}$ is a uniformly random vertex in $N(Y_t)$. 
The (empirical) speed of $(Y_t)$ is the process $(t^{-1}|Y_t|)$.
In contrast with $(X_t)$, the permuted random walk may not have a limiting speed.
The lower speed of the permuted random walk is defined by $\liminf_{t\to\infty} t^{-1} |Y_t|$.

Permuted random walk have been studied before, both on their own merit, and as a tool towards other ends.
Pymar and Sousi \cite{PySo} established uniform bounds on hitting times for permuted random walks on finite regular graphs.
Ganguly and Peres \cite{GaPe} studied walks on an interval with a fixed uniform random permutation. Recently, Chatterjee and Diaconis \cite{ChaDi1,ChaDi2} demonstrated that mixing of certain Markov chains can be significantly sped up by adding a deterministic permutation after each move.
In a different direction, Gou\"ezel \cite{Gouezel} used permuted random walks to establish large deviation lower bounds on the speed of random walks on hyperbolic spaces without moment assumptions on the step distribution.
One idea here is to condition on the long steps of the walk, and consider the process as a permuted version of a walk with bounded steps for which other methods can apply.
The question at the heart of this paper arose following a presentation of that work.

Our main result is that no matter how we select the permutations $(\pi_t)$,
the permuted walk $(Y_t)$ is not slower than $(X_t)$.

\begin{theorem}
  \label{thm:lazy}
  For every $d \geq 2$, every sequence $(\pi_t)$ of permutations of~$V(\T_d)$, and every time $t \ge 0$, the depth of the permuted random walk $|Y_t|$ stochastically dominates the depth of the lazy random walk $|X_t|$.
  That is, for all $t,n \geq 0$,
  \[  \Pr[|Y_t| \geq n] \geq \Pr[|X_t|\geq n].  \]
  In particular, 
  $\E |Y_t| \geq \E |X_t|$ for all $t \geq 0$,
  and $\liminf_{t \to \infty} t^{-1} |Y_t| \geq \frac{d-2}{d+1}$ almost surely.
\end{theorem}

\begin{remark}
  The proesses $(X_t)$ and $(Y_t)$ in \cref{thm:lazy} correspond to a lazy random walk that stays put with probability $\frac{1}{d+1}$.
  \cref{thm:lazy} holds verbatim (with the obvious change to the constant $\frac{d-2}{d+1}$) as long as the probability to stay put is at least~$\frac{1}{d+1}$.
  In particular, it holds when the chance to stay put is one half, which is a more common definition of the lazy random walk.
  For details, see the remark after the proof of \cref{thm:lazy}.
\end{remark}

Note that the theorem is informative even for $d=2$, where the limit speed is zero.
However, the laziness is required for \cref{thm:lazy} to hold.
Indeed, for the non-lazy walk on $\T_d$ we can have $\E |Y_1| < \E |X_1|$ (or for any other $t$).

\cref{thm:lazy} is a special case of a more general phenomenon, which we describe in the next two theorems.
For a distribution $p$ on $V$, define $p^*:\N \to [0,1]$ by letting $p^*(j)$ be the total mass of the $j$ largest atoms in $p$, or equivalently,
\[ p^*(s) = \max \{p(J) : J \subset V, |J|=s\}. \]
We say that a distribution $p$ \textbf{majorizes} a distribution $q$ if $p^*(j) \ge q^*(j)$ for all $j\in\N$.

Denote by $p_t$ the distribution of $X_t$ and by $q_t$ the distribution of $Y_t$, depending implicitly on the fixed permutations $(\pi_t)$.
The stochastic domination asserted in \cref{thm:lazy} is a consequence of the following more technical statement.
The main reasons are that the distribution $p_t$ is spherically symmetric and monotone in depth; 
for more details, see \cref{sec:PuttingLazy}.

\begin{theorem}
  \label{thm:moreCenterWithT}
  For every $t \geq 0$, the distribution $p_t$ majorizes $q_t$.
\end{theorem}

The fact that $p_t$ majorizes $q_t$ can be interpreted as saying
that the amount of disorder in $q_t$ is at most that of $p_t$.
Concretely, the theorem implies that the Shannon entropy of $Y_t$ is at most the Shannon entropy of $X_t$.
There is no way to increase the entropy of a lazy random walk on a regular tree by applying time dependent permutations. 

A second interpretation of the theorem is that for every $t$,
there is a distribution $r_t$ on permutations of $V$,
so that if $\sigma_t$ is sampled from $r_t$ independently of $X_t$, then $(X_t, \sigma_t(X_t))$ has the same distribution as $(X_t,Y_t)$.
In other words, there is a distribution on a single permutation $\sigma_t$
that allows to replace the iterative application of the $t$ permutations 
$\pi_1,\ldots,\pi_t$.

An even more general statement than \cref{thm:moreCenterWithT} holds.
Let $B_n = \{v \in V: |v| \leq n\}$ denote the ball of radius $n \geq -1$ in the tree\footnote{The ball $B_{-1}$ is empty.} and $\partial B_n=B_n \setminus B_{n-1}$ the sphere of radius $n$.
Fix an order $v_0,v_1,v_2,\ldots$ of $V$ with the following property:
for every $i < j$, it holds that $|v_i| \leq |v_j|$
and if $|v_i|=|v_j|$ then the $d-1$ children of $v_i$ appear
in the order before the $d-1$ children of $v_j$. 
Initial segments of the form $\{v_0,v_1,\dots,v_i\}$ are called \textbf{quasi-balls}.
Note that every ball is a quasi-ball.
A distribution $p$ on $V$ is called \textbf{greedily arranged} if $p(v_i) \ge p(v_{i+1})$ for every $i$.

\begin{theorem}
  \label{thm:moreCenterWithT2}
  Let $p$ and $q$ be distributions on $V$ and let $p'$ and $q'$ be the corresponding distributions after a single step of a lazy random walk started at $p$ and $q$, respectively.
  If $p$ is greedily arranged and majorizes $q$, then $p'$ is greedily arranged and majorizes $q'$.
\end{theorem}

\cref{thm:lazy,thm:moreCenterWithT} follow by a simple inductive argument from the last theorem using the following two observations. 
First, the initial distribution $p_0$ is greedily arranged, and majorizes $q_0$.
Second, if a distribution $p$ majorizes $q$, then it also majorizes any rearrangement of $q$ (i.e., a distribution of the form $q \circ \pi$ for a permutation $\pi$ of $V$).
Thus \cref{thm:moreCenterWithT2} implies that for every $t \ge 0$ and every finite $J \subset V$,
\begin{equation}\label{eq:pt_qt}
  p_t(B) \ge q_t(J) ,
\end{equation}
where $B$ is the quasi-ball of size $|B|=|J|$.

\subsection{Non-lazy random walks}

The last result is particular to lazy random walks on regular trees.
For non-lazy walks, it is too strong to be true.
The distribution of a simple non-lazy random walk on a regular tree is not greedily arranged because the tree is bipartite; in particular,~\eqref{eq:pt_qt} may fail already for $t=1$. 

On the other hand, versions of the above theorems do hold for non-lazy walks, as we describe next.
The limit speed of a simple (non-lazy) random walk on $\T_d$ is $\frac{d-2}{d}$ a.s.
As noted, for such random walks, the same stochastic domination as in \cref{thm:lazy} does not hold.
Nonetheless, we prove that it almost holds (at least for $d>2$, when the tree is not the line).

Denote by $N'(v)$ the $d$ neighbors of $v$ not including $v$.
Let $(S_t)$ be a simple random walk so that $S_0 = v_0$ and $S_{t+1}$ is uniform in $N'(S_t)$.
Let $(Z_t)$ be a permuted simple random walk so that $Z_0 = v_0$ and $Z_{t+1}$ is $\pi_{t+1}(Z'_{t+1})$, where $Z'_{t+1}$ is uniform in $N'(Z_t)$. 

\begin{theorem}
  \label{thm:simple}
  For every $d>2$, every sequence $(\pi_t)$ of permutations of $V(\T_d)$, and every time $t \ge 1$, we have that $|Z_t|+2$ stochastically dominates $|S_t|$.
  In particular, 
  $\E |Z_t| \geq \E |S_t|-2$ for all $t \geq 0$, and $\liminf_{t \to \infty} t^{-1} |Z_t| \geq \frac{d-2}{d}$ almost surely.
\end{theorem}

For $d=2$, the bound $\frac{d-2}{d}$ on the lower speed of $(Z_t)$ trivially holds, but the stronger claim in the theorem is false.
One way to see this is to take $\pi_t$ to be the identity up to some large time $2T$, 
and then map via $\pi_{2T}$ all even integers in the range $[-2T,2T]$ to all integers in $[-T,T]$ so that $\E |Z_{2T}| = \tfrac{1}{2} \E |S_{2T}|$.

We shall deduce \cref{thm:simple} from the following modification of \cref{thm:moreCenterWithT2} which takes into account the periodicity of the non-lazy walk.
The vertex set can be partitioned according to parity into
$V_0 = \{ v \in V : |v|=0 \mod 2\}$ and $V_1 = V \setminus V_0$. 
A distribution $p$ is called \textbf{half-greedily arranged} if it is supported on one of $V_0$ or $V_1$, and $p(v_i) \ge p(v_j)$ for every $i<j$ for which $v_i$ and $v_j$ have the same parity (using the same ordering of $V$ as above).

\begin{theorem}
  \label{thm:moreCenterWithT-nonlazy}
  Let $p$ and $q$ be distributions on $V$, and let $p'$ and $q'$ be the corresponding distributions after a single step of a non-lazy random walk started at $p$ and $q$, respectively.
  If $p$ is half-greedily arranged and majorizes $q$, then $p'$ is half-greedily arranged and majorizes $q'$.
\end{theorem}

Although the distribution of $S_t$ is not greedily arranged, it is half-greedily arranged. 
The theorem thus implies that the distribution of $S_t$ majorizes that of $Z_t$ for every $t \ge 0$
(although the distribution of $|S_t|$ does not necessarily majorizes that of $|Z_t|$).

\subsection{The speed process}

\cref{thm:lazy,thm:simple} establish stochastic domination of the distance of a standard (lazy/simple) random walk over the distance of a permuted random walk \emph{at any particular time}.
It is natural to wonder whether such stochastic domination holds for the corresponding processes, i.e., whether the two processes can be coupled so that the distance of the permuted walk is always at least the distance of the standard walk.
Somewhat surprisingly, it turns out this is not always possible.
We focus on lazy random walks for concreteness.
As an example, consider a sequence of permutations $\pi$ in which $\pi_1$ and $\pi_2$ are the identity permutation and $\pi_3$ is an automorphism of $T$ which maps a neighbor of $v_0$ to $v_0$.
A direct computation yields that
\[ \Pr[|X_2|+|X_3|\le 2] = \tfrac1{d+1}+\tfrac{4d}{(d+1)^3} < \tfrac1{d+1}+\tfrac{5d-1}{(d+1)^3} = \Pr[|Y_2|+|Y_3|\le 2] ,\]
so that $(|Y_2|,|Y_3|)$ does not stochastically dominate $(|X_2|,|X_3|)$.

When $d=2$, this effect can be repeated and magnified over time.
The next result shows that for certain choices of permutations, even translations, there are infinitely many times at which the distance of the permuted random walk is much smaller (no matter how the two processes are coupled).

\begin{theorem}
  \label{thm:lazy-bad}
  Fix $d=2$. There exists a sequence of permutations $(\pi_t)$ of~$V(\T_2) \cong \Z$, all of which are translations, such that in any coupling of the lazy random walk process $(X_t)$ and the permuted random walk process $(Y_t)$, almost surely,
  \begin{equation}\label{eq:limsup}
    \limsup_{t \to \infty} \frac{|X_t|-|Y_t|}{\sqrt{t\log\log t}} = \frac{\sqrt3}2 .
  \end{equation}
\end{theorem}


When $d>2$, on the other hand, we show that the above cannot occur (not even nearly) when the permutations are required to be automorphisms of $\T_d$.
This is the content of the result below.
We do not know how strong this effect can be for general permutations.
For instance, we do not know whether it is always possible to couple the two processes so that, almost surely, $|X_t| \le |Y_t|$ for all large enough $t$. 

\begin{theorem}
  \label{thm:lazy-bad2}
  For every $d>2$ and every sequence of automorphisms $(\pi_t)$ of~$\T_d$, there exists a coupling of the lazy random walk process $(X_t)$ and the permuted random walk process $(Y_t)$ such that, almost surely,
  \[ |Y_t| - |X_t| \ge t^{1/2 - o(1)} \qquad\text{as }t\to\infty .\]
\end{theorem}

The theorem is interesting even when each $\pi_t$ is the identity.
It states that there is a way to couple two lazy random walks so that one is significantly more distant than the other. 
The result is tight is the sense that the $o(1)$ term cannot be dropped entirely.
Our proof gives a quantitative estimate for this term and yields that $t^{1/2-o(1)}$ can be replaced with $\sqrt t/(\log^C t)$ for some constant $C>0$.
See \cref{lem:SRW-coupling} and the second remark following it.

\subsection{A spectral argument}

One natural approach towards proving the results above is using spectral methods (see~\cite{morris2005evolving} and references within).
Specifically, the transition kernel on $\ell^2(V)$ is a contraction with norm $\rho<1$, and application of a permutation is an isometry on $\ell^2(V)$.
Thus $\|q_t\|_2 \leq \rho^t$ decays exponentially.
A positive lower bound on the lower speed of $Y_t$ follows easily.
Moreover, this argument holds for any non-amenable graph.
However, the resulting bound on the speed is not sharp.

The proof of a spectral gap uses an isoperimetric inequality for the tree.
Not surprisingly, our proofs also use isoperimetric inequalities;
see \cref{prop:iso,prop:iso-exact} below.
\cref{prop:iso} is a non-standard isoperimetric inequality, which takes into account the amount of ``isolated'' points in the set of interest.
\cref{P:ki-mi-inequality} is a significant generalization of the isoperimetric inequality using the language of majorization.

\section{Isoperimetry}

As noted, our arguments rely on isoperimetric properties of the tree.
However, to get the strongest possible comparison between the permuted and regular random walks we need sharp isoperimetric inequalities, which we now proceed to prove.

Recall that $N(v)$ is the neighborhood of a vertex $v$, including $v$ itself.
For $J \subset V$, the neighborhood of $J$ is defined by
\[N(J) = \bigcup_{v \in J} N(v).\]
To analyze the behavior of the random walk, we need to understand the boundary in more detail.
For $J \subset V$ and $i \in [d+1]$, define
\[ K_i(J) = \{ v \in V : |N(v) \cap J| \ge i \} .\]
In particular, the set $K_1(J)$ is the neighborhood $N(J)$.

A \textbf{partition} is a sequence $\mu = (\mu_1,\mu_2,\dots,\mu_\ell)$ with $\mu_1 \geq \mu_2 \geq \dots \geq \mu_\ell \geq 0$.
Note that usually trailing 0's are omitted, but for us it is convenient to have the length of the partitions be fixed, so we may include 0's. 
The size of the partition is defined by $|\mu| = \sum_i \mu_i$.
The dominance order on partitions is defined as follows.
For partitions $\mu,\lambda$, we write $\lambda \prec \mu$ if $|\mu| = |\lambda|$ and
\begin{equation}\label{eq:dominance}
  \lambda_1 + \dots+\lambda_r \leq \mu_1+\dots+\mu_r,
  \qquad \text{ for all $r$}.
\end{equation}

The following majorization statement is an extension of the standard isoperimetric inequality for the tree.

\begin{prop}\label{P:ki-mi-inequality}
  Let $J \subset V$ be finite and let $B$ be the quasi-ball with $|B|=|J|$.
  Let $k_i=|K_i(J)|$ and $m_i=|K_i(B)|$.
  Then $(k_i)$ dominates $(m_i)$ as partitions:
  $(k_1,\dots,k_{d+1}) \succ (m_1,\dots,m_{d+1})$.
\end{prop}

To prove this result, we need a couple of lemmas on the isoperimetric behavior of the tree.
Let $\kappa_1(J)$ denote the number of connected components induced by $J$.
Let $\kappa_2(J)$ denote the number of connected components induced by $J$ in the graph in which edges are added between all pairs of vertices that are at distance 2 from each other in the tree.
The first lemma is a formula for $|N(J)|$ for general $J$:

\begin{prop}\label{prop:iso-exact}
  For every finite $J \subset V$,
  \[ |N(J)| = (d-1)|J| + \kappa_1(J) + \kappa_2(J) .\]
\end{prop}

\begin{proof}
We prove the claim by induction on $|J|$. The base case when $|J|=0$ is trivial. Let $J$ be non-empty. Let $v \in J$ be 
a vertex of maximum depth in $J$.
Let $N_1=N(v) \cap (J \setminus \{v\})$ and $N_2=N(N(v)) \cap ( J \setminus \{v\})$.  
The following two equalities hold:
\[ \kappa_1(J \setminus \{v\})=\kappa_1(J)-\mathbf{1}_{\{N_1=\emptyset\}} \qquad
\text{and}  \qquad \kappa_2(J \setminus \{v\})=\kappa_2(J)-\mathbf{1}_{\{N_2=\emptyset\}}.\]
The induction hypothesis implies
\[ |N(J \setminus \{v\})|= (d-1)|J| + \kappa_1(J) + \kappa_2(J) - (d-1 + \mathbf{1}_{\{N_1=\emptyset\}} + \mathbf{1}_{\{N_2=\emptyset\}}) .\]
It remains to show that
\[ |N(J)|-|N(J \setminus \{v\})|= d-1 + \mathbf{1}_{\{N_1=\emptyset\}} + \mathbf{1}_{\{N_2=\emptyset\}} .\]
The left-hand side equals 
$$|N(J) \setminus N(J \setminus \{v\})|
= |N(v) \setminus N(J \setminus \{v\})| = d+1 - |N(v) \cap N(J \setminus \{v\})|.$$ 
So we need to show that
\[ |N(v) \cap N(J \setminus \{v\})|= 2 - \mathbf{1}_{\{N_1=\emptyset\}} + \mathbf{1}_{\{N_2=\emptyset\}}
= \mathbf{1}_{\{N_1 \neq \emptyset\}} + \mathbf{1}_{\{N_2\neq\emptyset\}} .\]
By the choice of $v$,
there are at most two vertices in $N(v) \cap N(J \setminus \{v\})$;
the vertex $v$ and its parent.
The vertex $v$ is in $N(J \setminus \{v\})$ iff $N_1 \neq \emptyset$.
Its parent is in $N(J \setminus \{v\})$ iff $N_2 \neq \emptyset$.
\end{proof}

For the next lemma, we also need the following definitions.
The sets of isolated points in $J$ and connected points in $J$ are defined  by
\begin{align*}
  \iso(J) = \{ v \in J : N(v) \cap J =\{v\}\} \quad \text{and} \quad
  \conn(J) = J \setminus \iso(J).
\end{align*}

\begin{prop}
  \label{prop:iso}
  For every non-empty $J \subset V$, 
  \[|N(J)| \geq \begin{cases}
      1 + d|J| & \conn(J) = \emptyset \\
      2 + d |\iso(J)| + (d-1)|\conn(J)| & \conn(J) \neq\emptyset.
    \end{cases}
  \]
\end{prop}

\begin{proof}
  Using \cref{prop:iso-exact},
  \begin{align*}
    |N(J)| 
    & = (d-1)(|\iso(J)|+|\conn(J)|) + \kappa_1(J) + \kappa_2(J) \\
    & \geq (d-1)(|\iso(J)|+|\conn(J)|) + (|\iso(J)| + \mathbf{1}_{\{\conn(J)\neq\emptyset\}})
    + 1 . \qedhere
  \end{align*}
\end{proof}

\begin{proof}[Proof of \cref{P:ki-mi-inequality}]
  The fact that $(k_i)$ and $(m_i)$ are decreasing is obvious.
  These are partitions of the same size $s = (d+1)|J|$.
  If $|J|=1$ then the statement trivially holds, so we can assume $|J|>1$.
  The choice of order on $V$ implies there is $n \geq 0$ so that
  $B_n \subseteq B \subsetneq B_{n+1}$, where $B_n$ is the ball of radius $n$.
  We can write
  \[ |J| = |B| = |B_n| + a(d-1) + c ,\]
  where $a,c$ are non-negative integers so that $c < d-1$.

  The tree is simple enough so that we can compute all the $m_i$'s in terms of these:
  \begin{align*}
    m_1 & = (d-1)|J| + 2 , \\
    m_2 &  = |J| , \\
    \forall \ 3 \leq i \leq c+2 \qquad
    m_i & = |B_{n-1}| + a + 1 , \\
    \forall \ c+3 \leq i \leq d+1 \qquad
    m_i & = |B_{n-1}| + a . 
  \end{align*}
  The case $r=1$ of \eqref{eq:dominance} now holds by \cref{prop:iso-exact}:
  \[ k_1 = |N(J)| \ge (d-1)|J| + 2 = m_1 .\]
  The case $r=2$ is proved as follows.
  If $k_2 = 0$ then $k_i = 0$ for all $i \geq 2$ and the proof is complete.
  On the other hand, if $k_2 \geq 1$ then by \cref{prop:iso}, and because $\conn(J) \subseteq K_2(J)$,
\[ k_1+k_2  \geq 
(d-1)|J| + |\iso(J)| + \mathbf{1}_{\{\conn(J)\neq\emptyset\}}
 + 1 + k_2 \geq d|J|+2 = m_1+m_2 .\]
For $r \in \{3,4,\ldots,d\}$, proceed by induction. 
Because $k_r \geq k_{r+1} \geq \ldots \geq k_{d+1}$, we have
$$k_r \geq a_r:= \frac{s - \sum_{i \in [r-1]} k_i}{d-r+2}.$$
By induction,
\begin{align*}
\sum_{i \in [r]} k_i & \geq a_r + 
\sum_{i \in [r-1]} k_i \\
& = \frac{s}{d-r+2} + \frac{d-r+1}{d-r+2} \sum_{i \in [r-1]} k_i \\
& \geq \frac{s}{d-r+2} + \frac{d-r+1}{d-r+2} \sum_{i \in [r-1]} m_i  \\
& = b_r  + \sum_{i \in [r-1]} m_i ,
\end{align*}
where 
\begin{align*}
b_r 
:= \frac{s - \sum_{i \in [r-1]} m_i}{d-r+2} .
\end{align*}
If $r \leq c+2$, then
\begin{align*}
b_r 
& = \frac{(c+2-r+1)(|B_{n-1}| + a + 1)
+ (d+1-c-2)(|B_{n-1}| + a)}{d-r+2} \\
& = \frac{(d-r+2)(|B_{n-1}| + a) + c+2-r}{d-r+2} ,
\end{align*}
and if $r > c+2$, then
\begin{align*}
b_r 
= \frac{(d-r+2)(|B_{n-1}| + a)}{d-r+2} .
\end{align*}
It follows that $m_r = \lceil b_r \rceil$.
All $k_i$'s and $m_i$'s are integers, so the desired inequality follows.
\end{proof}

\section{Lazy random walks}
\label{sec:PuttingLazy}

The following proposition presents the key link between the isoperimetric inequality and the behavior of random walks.

\begin{prop}\label{prop:q*}
  Let $J \subset V$ be finite and let $B$ be the quasi-ball with $|B|=|J|$.
  Let $k_i=|K_i(J)|$ and $m_i=|K_i(B)|$. 
  For every distribution $q$,
  \[ \sum_{i \in [d+1]} q^*(k_i) \le \sum_{i \in [d+1]} q^*(m_i) .\]
\end{prop}

This may seem surprising until one realizes that $q^*$ can be any function on $\N$ that is increasing from $0$ to $1$ and is concave.
The proof of \cref{prop:q*} is based on the following majorization inequality, known as the Hardy--Littlewood--P\'olya inequality and Karamata's inequality, a version of which was first proved by Schur;
see e.g. \cite[Theorem~3.C.1]{marshall1979inequalities}.
Note that the definition of the dominance order $\mu\succ\lambda$ extends verbatim to partitions of a real number with real instead of integer parts, and so this applies also for non-integer dominated sequences.
In our setting, $k_i$ and $m_i$ are integers.

\begin{theorem}\label{thm:conc} 
  Let $I\subset \R$ be an interval and let $f:I \to \R$ be concave.
  If $\mu,\lambda \in I^t$ are two partitions such that $\mu \succ \lambda$, then
  \[ \sum_{i \in [t]} f(\mu_i) \leq \sum_{i \in [t]} f(\lambda_i). \]
\end{theorem}

\begin{proof}[Proof of \cref{prop:q*}]
  To apply \cref{thm:conc,P:ki-mi-inequality} we need to extend $q^*$ to a concave function.
  By construction, the function $q^*:\N \to [0,1]$ is increasing and can be written as $q^*(j) = \sum_{i \in [j]} D(i)$ where $D : \N \to [0,1]$ is a decreasing function.
  Thus extending $q^*$ to $\R_+$ by a piecewise linear interpolation is increasing and concave.
\end{proof}

The following observation helps to establish the property that a distribution is greedily arranged. 

\begin{obs}\label{obs:quasi-ball}
  Let $B$ be a quasi-ball and let $i \in [d+1]$.
  Then, $K_i(B)$ is a quasi-ball. 
\end{obs}

\begin{proof}
Write $B$ as $B = \{v_0,v_1,\ldots,v_i\}$. 
The choice of order on $V$ implies
there is $n \geq 0$ so that
$B_n \subseteq B \subsetneq B_{n+1}$, and we can write
$|B| = |B_n| + a(d-1) + c$,
where $a,c$ are non-negative integers so that $c < d-1$.
Analyze the different $K_i(B)$'s as follows.
The set $K_1(B) = N(B)$ contains $B_{n+1}$ 
and some of the smallest elements in $\partial B_{n+2}$.
The set $K_2(B)$ is equal to $B$.
For $i \in \{3,\ldots,c+2\}$,
the set $K_i(B)$ contains $B_{n-1}$
and the $a+1$ smallest elements in $\partial B_n$.
For $i \in \{c+3,\ldots,d+1\}$,
the set $K_i(B)$ contains $B_{n-1}$
and the $a$ smallest elements in $\partial B_n$.
\end{proof}

We are now ready to complete the proof of our main results.

\begin{proof}[Proof of \cref{thm:moreCenterWithT2}]
  Let $J \subset V$ and let $B$ be a quasi-ball of the same size.
  For $i \in [d+1]$, let $k_i=|K_i(J)|$ and $m_i=|K_i(B)|$.
  We have
  \begin{align}
    q'(J) &= \frac{1}{d+1} \sum_{i \in [d+1]} q(K_i(J)) \nonumber \\
    & \leq \frac{1}{d+1} \sum_{i \in [d+1]} q^*(k_i) \label{eq:s1} \\
    & \leq \frac{1}{d+1} \sum_{i \in [d+1]} q^*(m_i) \label{eq:s2} \\
    & \leq \frac{1}{d+1} \sum_{i \in [d+1]} p^*(m_i) \label{eq:s3} \\
    & = \frac{1}{d+1} \sum_{i \in [d+1]} p(K_i(B)) \label{eq:s4} \\
    & = p'(B). \nonumber
  \end{align}
  Here, the first and last equalities follow from the definition of the lazy random walk;
  \eqref{eq:s1} follows from the definition of $q^*$;
  \eqref{eq:s2} follows from Proposition~\ref{prop:q*}; \eqref{eq:s3} holds because $p$ majorizes $q$;
  finally, \eqref{eq:s4} follows from Observation~\ref{obs:quasi-ball} and the assumption that $p$ is greedily arranged. 

  For the set $J$ that achieves $q'^*(s)$, the above implies that 
  ${q'}^*(s) \leq p'(B) = {p'}^*(s)$.
  The fact that $p'$ is greedily arranged follows from Observation~\ref{obs:quasi-ball}. 
\end{proof}

\begin{proof}[Proof of \cref{thm:lazy}]
  \cref{thm:moreCenterWithT2} implies~\eqref{eq:pt_qt} and in particular $p_t(B_n) \ge q_t(B_n)$ for all $n,t \ge 0$.
  In other words, $|Y_t|$ stochastically dominates $|X_t|$ for every $t \ge 0$.
  This implies that $\E |Y_t| \ge \E |X_t|$.
  It remains to show that $\liminf_{t \to \infty} t^{-1} |Y_t| \ge \frac{d-2}{d+1}$ almost surely.
  For every $\eps > 0$, standard concentration bounds show that for some constants $c,C>0$, 
  \[ \Pr\big[ t^{-1} |X_t| < \tfrac{d-2}{d+1} - \eps \big] \leq C e^{-ct}. \]
  Since $|Y_t|$ stochastically dominates $|X_t|$ for every $t \ge 0$, the same holds with $Y_t$ instead of $X_t$.
  The Borel--Cantelli lemma completes the proof.
\end{proof}

\begin{remark}
  \cref{thm:moreCenterWithT2}, and thus also \cref{thm:lazy,thm:moreCenterWithT}, extends to the lazy random walk in which the probability to stay put is any $\gamma \geq \frac{1}{d+1}$.
  The idea is that if $q'_\gamma$ is the result of a lazy random walk step applied to a distribution $q$ with lazyness $\gamma$, then for any $\gamma>\delta$,
  \[ q'_{\gamma} = \big(\gamma-\delta\big) q + \big(1-\gamma+\delta\big) q'_\delta. \]
  We apply this with $\gamma > \delta = \frac{1}{d+1}$ to get
  \begin{align*}
    q'_\gamma(J) 
    & = \Big(\gamma - \frac{1}{d+1} \Big) q(J) +
    \Big(1-\gamma+ \frac{1}{d+1} \Big)  \frac{1}{d+1} \sum_{i\in [d+1]} q(K_i(J)) \\
    & \leq \Big(\gamma - \frac{1}{d+1} \Big) p(B) +
    \Big(1-\gamma+ \frac{1}{d+1} \Big)  \frac{1}{d+1} \sum_{i\in [d+1]}  p(K_i(B)) = p'(B).
  \end{align*}
\end{remark}

\section{Simple random walks}
\label{sec:Simple}

In this section, we consider simple (non-lazy) walks.
The argument is similar to the lazy case, and we omit some of the details that are unchanged.
For $J \subset V$, let
\[ N'(J) = \bigcup_{v \in J} N'(v). \]
The main difficulty stems from the fact that the tree is bipartite. 
The \textbf{half-ball} $B'_n$ is the set of the form
\[ B'_n = \{ v \in B_n : |v| \equiv n \mod 2\}. \]
A \textbf{half-quasi-ball} is the intersection of a quasi-ball with either $V_0$ or with $V_1$.
Half-qausi-balls have parities.
A \textbf{half-greedily arranged} distribution is a distribution supported on a quasi-ball.

\begin{prop}
  \label{prop:isoSimple}
  For every non-empty $J \subset V$,
  \[ |N'(J)| \geq 1 + (d-1) |J|. \]
\end{prop}

\begin{proof}
First assume that $J$ is contained in either $V_0$ or $V_1=V \setminus V_0$.
In this case, $\kappa_1(J)=|J|$ so that Proposition~\ref{prop:iso-exact} implies that
\[ |N'(J)| = |N(J)| - |J| = (d-1)|J| + \kappa_2(J) \ge (d-1)|J| + 1 .\]
Second, for arbitrary $J$, we have $N'(v) \cap N'(w) = \emptyset$ if $|v| \neq |w| \mod 2$.
The result follows by applying the above to $J \cap V_0$ and $J \cap V_1$ separately. 
\end{proof}

For $J \subset V$ and $i \in [d]$, define
\[ K'_i(J) = \{ v \in V : |N'(v) \cap J| \ge i \} .\]
Fix $J$ and let $B$ be a half-quasi-ball of the same size.
Let $k_i=|K'_i(J)|$ and $m_i=|K'_i(B)|$.

\begin{prop}\label{prop:q*2}
For any distribution $q$ on $V$, 
\[ \sum_{i \in [d]} q^*(k_i) \le \sum_{i \in [d]} q^*(m_i) .\]
\end{prop}

\begin{proof}
As before, the proposition follows from \cref{thm:conc} once we show that 
$(k_1,\ldots,k_{d}) \succ (m_1,\ldots,m_d)$.
The fact that these are partitions is obvious, and they have the same size since
$d|J| = \sum_{i \in [d]} k_i = \sum_{i \in [d]} m_i$.
It remains to establish~\eqref{eq:dominance} for these partitions. 
Write
\[ |J| = |B| = |B'_n| + a(d-1) + c ,\]
where $B'_n \subset B \subsetneq B'_{n+2}$,
and $a,c \geq 0$ are integers so that $c  \leq d-2$.
The values of the $m_i$'s are now as follows: $m_1 = (d-1)|J| + 1$,
for $2 \leq i \leq c+1$, we have
$m_i = |B'_{n-2}| + a + 1$,
and for $c+2 \leq i \leq d$, we have
$m_i = |B'_{n-2}| + a$.
The inequality $\sum_{i \in [r]} k_i \ge \sum_{i \in [r]}m_i$ for $r=1$ follows from \cref{prop:isoSimple}.
For $r \ge 2$, one proceeds by induction in a similar manner as in the proof of \cref{P:ki-mi-inequality}.
\end{proof}

\begin{obs}\label{obs:half-quasi-ball}
  Let $B$ be a half-quasi-ball, and $i \in [d]$.
  Then, $K'_i(B)$ is a half-quasi-ball of opposite parity than $B$.
\end{obs}

\begin{proof}[Proof of \cref{thm:moreCenterWithT-nonlazy}]
Let $J \subset V$ and let $B$ be the half-quasi-ball of the same size as $J$
and with opposite parity than $p$.
Let $k_i=|K'_i(J)|$ and $m_i=|K'_i(B)|$.
We have
\begin{align}
q'(J) &= \frac{1}d \sum_{i \in [d]} q(K'_i(J)) \nonumber \\
& \leq \frac{1}d \sum_{i \in [d]} q^*(k_i) \label{eq:t1} \\
& \leq \frac{1}d \sum_{i \in [d]} q^*(m_i) \label{eq:t2} \\
& \leq \frac{1}d \sum_{i \in [d]} p^*(m_i) \label{eq:t3} \\
& = \frac{1}d \sum_{i \in [d]} p(K'_i(B)) \label{eq:t4} \\
& = p'(B), \nonumber
\end{align}
where the first and last equalities follow from the definition of the non-lazy random walk; \eqref{eq:t1} follows from the definition of $q^*$; \eqref{eq:t2} follows from \cref{prop:q*2}; \eqref{eq:t3} holds because $p$ majorizes $q$; and~\eqref{eq:t4} follows from \cref{obs:half-quasi-ball} and the assumption that $p$ is half-greedily arranged. The result follows in the same way as in the proof of \cref{thm:moreCenterWithT2}. 
\end{proof}

\begin{proof}[Proof of \cref{thm:simple}]
  Denote by $p_t$ the distribution of $S_t$, and denote by $q_t$ the distribution of $Z_t$.
  Since $d>2$, we have $|B_n| \leq 1 + d (d-1)^{n} \leq |B'_{n+1}|$.
  We then have
  \begin{align}
    q_t(B_n)
    & \leq q^*_t(|B_n|) \label{eq11} \\
    & \leq p^*_t(|B_n|) \label{eq12}\\
    & \leq p^*_t(|B'_{n+1}|) \label{eq13}\\
    & \leq p_t(B_{n+2}) , \label{eq15}
  \end{align}
  where~\eqref{eq11} holds by definition of $q_t^*$;
  \eqref{eq12}~holds by \cref{thm:moreCenterWithT-nonlazy} 
  and induction on $t$;
  \eqref{eq13}~holds because $|B_n| \leq |B'_{n+1}|$;
  and \eqref{eq15}~holds because $p_t$ is half-greedily arranged,
  and because $B'_{n+1} \cup B'_{n+2} \subseteq B_{n+2}$.
  The rest of the proof proceeds in a similar manner as in the proof of \cref{thm:lazy}.
\end{proof}

\section{Exceptional times}

In this section we consider the possible slow-down of a random walk on $\Z \cong \T_2$ and on $\T_d$ for $d>2$.
While the domination of \cref{thm:lazy} still applies, we ask here whether $(\pi_t)$ may be chosen so that there are exceptional times where $|Y_t|$ is much smaller than $|X_t|$.
We prove \cref{thm:lazy-bad} on the {\em existence} of exceptional times of slowing down on~$\Z$.
In contrast, we prove \cref{thm:lazy-bad2} on the {\em non-existence} of such times on $T_d$ when $d>2$ and the permutations are restricted to automorphisms.
This section is mostly independent of the previous parts of the paper.

\subsection{Exceptional times for $\Z$}
  
\begin{proof}[Proof of Theorem~\ref{thm:lazy-bad}]
  The permutations $(\pi_t)$ are all translations of $\Z$.
  Consequently, the permutations commute not just with each other but with the steps of the random walk.
  We shall define an integer sequence $\ell_t$, and define the permutations $\pi_t$ by $\pi_t \pi_{t-1} \cdots \pi_1(v)=v-\ell_t$.
  Thus the process $(Y_t+\ell_t)$ has the same law as the random walk $(X_t)$.
However, the coupling between the processes may not be such that $Y_t = X_t-\ell_t$, even though that is one possible coupling.

  To define $(\ell_t)$, let $\phi(t)$ denote the integer part of $(\frac43 t \log\log t)^{1/2}$.
  Let $f(t)$ be a positive integer-valued non-decreasing function growing to infinity slower than $\phi(t)$.
  Let $(b_j)_{j=0}^\infty$ be defined by $b_0 = 1$ and $b_{j+1} = b_j + f(b_j)$ for all $j\geq 0$.
  Let $(\ell_t)$ be defined by
  $\ell_{b_j+i}$ is the integer part of $\phi(b_j) \cdot (\frac{4i}{f(b_j)}-2)$ for all $j \ge 0$ and $0 \le i < f(b_j)$.
  Intuitively, for each $j$, the numbers of the form $\ell_{b_j+i}$
  are uniformly and densely placed in the interval between $-2 \phi(b_j)$ and $2 \phi(b_j)$.

  Fix $\eps>0$ and consider the set $T_\eps$ of times $t$ at which $X_t \ge (1-\eps) \phi(t)$. By the law of the iterated logarithm for the lazy random walk $(X_t)$, we have that $T_\eps$ is almost surely infinite.
  By the same law, almost surely, the set $T'$ of times $t$ at which $|Y_t+\ell_t| \le 1.5\phi(t)$ contains all but finitely many positive integers.

  Fix $t \in T_\eps \cap T'$ sufficiently large.
  Let $j \ge 0$ be such that $b_j \le t < b_{j+1}$.
  Since 
  $|Y_t+\ell_t| \le 1.5\phi(t) < 2\phi(b_j)$, there exists $0 \le i < f(b_j)$ such that 
  \[|(Y_t+\ell_t) - \ell_{b_j+i}| \le 
  \eps \phi(t).\] 
  At time $t' = b_j+i$, we have
  \[ X_{t'} \ge X_t - |t-t'| \ge (1-\eps)\phi(t) - f(b_j) \ge (1-\eps)\phi(t) - f(t) > 0 ,\]
  and
  \[ |Y_{t'}|=|(Y_{t'}+\ell_{t'})-\ell_{t'}| \le |(Y_t+\ell_t)-\ell_{t'}| + |t-t'| \le \eps \phi(t) + f(t) .\]
  Thus,
  \[ |X_{t'}|-|Y_{t'}| \ge (1-2\eps)\phi(t) - 2f(t) \ge (1-3\eps)\phi(t') .\]
  We conclude that almost surely,
  \[ \limsup_{t \to \infty} \frac{|X_t|-|Y_t|}{\phi(t)} \ge 1. \]
  Since $\limsup_{t \to \infty} \frac{|X_t|}{\phi(t)} = 1$ almost surely, we have equality above.
\end{proof}

\subsection{No exceptional times for $d>2$}

We split the proof of \cref{thm:lazy-bad2} into two parts for readability, and in order to emphasize the missing piece for lifting the automorphism restriction.

\begin{lemma}
  \label{thm:lazy-bad22}
  For every $d>2$ and every sequence of automorphisms $(\pi_t)$ of~$T_d$, there exists a coupling of the lazy random walk process $(X_t)$ and the permuted random walk process $(Y_t)$ such that, almost surely,
  \begin{equation*}
    |Y_t| \geq |X_t| - 2\log t \qquad\text{for all $t$ large enough}.
\end{equation*}
\end{lemma}

\begin{proof}
  Using that $(\pi_t)$ consists only of automorphisms, it is not hard to check that $(\pi_t\pi_{t-1}\cdots\pi_1 X_t)$ has the same distribution as the permuted random walk process $(Y_t)$.
  Thus, setting $Y_t=\pi_t\pi_{t-1}\cdots\pi_1 X_t$ describes a coupling between $(X_t)$ and $(Y_t)$.

  To see that this coupling satisfies the claimed property, note that $|X_t|-|Y_t|>k$ implies that either $|X_t|<k$ or $(\pi_t\pi_{t-1}\cdots\pi_1)^{-1} v_0 \in T_k(X_t)$, where $T_k(x)$, defined when $|x| \ge k$, is the connected component (subtree) of $\{ v \in V(\T) : |v| \ge k \}$ containing $x$.
  Since $X_t$ is uniform given its depth $|X_t|$, we see that
  \[ \Pr(|X_t|-|Y_t|>k) \le \Pr(|X_t|<k) + \frac1{|\partial B_k|} ,\]
  where $\partial B_k = \{ v \in V(\T) : |v|=k \}$ and $|\partial B_k|=d(d-1)^{k-1}$.
  Standard concentration bounds on the speed of $(X_t)$ now imply that $\sum_{t=1}^\infty \Pr(|X_t|-|Y_t|> 2\log t) < \infty$, and the Borel--Cantelli lemma completes the proof.
\end{proof}

\begin{lemma}\label{lem:SRW-coupling}
  Let $(X_t)$ be a non-trivial nearest-neighbor random walk on $\Z$ (possibly biased and with any laziness).
  There is a coupling of $(X_t)$ with another copy of itself $(X'_t)$ such that for some constant $C>0$, almost surely,
  \[ X'_t-X_t \ge \frac{\sqrt t}{(\log t)^C} \qquad\text{for all $t$ large enough}.\]
\end{lemma}

\begin{remark}
For positively biased random walks, $X_t$ and $X'_t$ are eventually positive so that the conclusion is equivalent to $|X'_t|-|X_t| \ge \sqrt t /(\log t)^C$. By interchanging the roles of $(X_t)$ and $(X'_t)$, the same statement is seen to hold also for negatively biased random walks. For unbiased random walks, on the other hand, it holds that $X'_t=0$ infinitely often.
\end{remark}

\begin{remark}
The term $\sqrt t/(\log t)^C$ is not optimal, but it cannot be improved to $c\sqrt t$. Indeed, in any coupling, the probability of the event $\{X_t \ge 0, X'_t < c\sqrt t\}$ is bounded from below, so that Fatou's lemma gives that $X'_t-X_t<c\sqrt t$ infinitely often with positive probability.
\end{remark}

\begin{proof}
We may always couple $(X_t)$ and $(X'_t)$ so that they stay put at the same times (and this set of times has density less than 1). It therefore suffices to handle the non-lazy case. We thus assume that $\Pr(X_1-X_0=1)=p$ and $\Pr(X_1-X_0=-1)=1-p$ for some $p \in (0,1)$.

The main step is to construct a coupling between two Binomial$(n,p)$ random variables $B_n$ and $B'_n$ such that
\[ \Pr\big(B'_n - B_n \ge \tfrac{\sqrt n}{\log^2 n} \big) \ge c' \qquad\text{and}\qquad \Pr(B'_n < B_n) \le \tfrac{C'}{\log^2 n} ,\]
where $c',C'>0$ are constants that depend on $p$ but not on $n$. Let $m$ be the integer part of $\sqrt n / \log^2 n$ and consider the two intervals
\[ I_n=[pn-\sqrt n,pn] \cap \N \qquad\text{and}\qquad J_n=[pn- \sqrt n,pn-m] \cap \N .\]
Denote $f(i)=\Pr(B_n=i)$ and observe that $\frac{f(i)}{f(i-1)} = \frac{p}{1-p} \cdot \frac{n-i+1}{i} \ge 1$ whenever $i \le p(n+1)$. Thus, $f(i)$ is increasing for $i \in I_n$, and $f(i) \le f(i+m)$ for $i \in J_n$. It follows that there is a coupling such that
\begin{align*}
 B_n \notin I_n &\implies B'_n=B_n,\\
 B_n \in J_n &\implies B'_n=B_n+m,\\
 B_n \in I_n \setminus J_n &\implies B'_n \in I_n .
\end{align*}
The central limit theorem implies that $\Pr(B_n \in J_n)$ converges as $n\to\infty$ to some positive constant $c=c(p)$. Since $f$ is bounded from above by $C/\sqrt n$ for some constant $C=C(p)$, we have that $\Pr(B_n \in I_n \setminus J_n) \le Cm/\sqrt n \le C/\log^2 n$.
This completes the construction of a coupling between $B_n$ and $B'_n$ with the claimed properties.

The above coupling between $B_n$ and $B'_n$ is relevant because $X_t$ has the same law as $2B_t-t$. Consider the times $t_n=2^n$ for $n \ge 1$. We construct the coupling between $(X_t)$ and $(X'_t)$ so that it is Markovian at these times. Fix $n \ge 1$ and suppose we have already coupled $(X_t)_{t \le t_n}$ and $(X'_t)_{t \le t_n}$ in some manner (the coupling for $n=1$ can be done arbitrarily). We now describe the (conditional) coupling between the processes in the time range $(t_n,t_{n+1}]$. This coupling only depends on $X_{t_n}$ and $X'_{t_n}$.
The law of $(X_t-X_{t_n})_{t_n \le t \le t_{n+1}}$ and $(X'_t-X'_{t_n})_{t_n \le t \le t_{n+1}}$ is entirely independent of the past (conditioned on time $t_n$). 
These are two random walks of length $t_{n+1}-t_n=2^n$, which we denote by $(S_i)_{i=0}^{2^n}$ and $(S'_i)_{i=0}^{2^n}$. To couple these walks, we first couple the endpoints $S_{2^n}$ and $S'_{2^n}$ using the above coupling between $B_{2^n}$ and $B'_{2^n}$ (pushed forward by the map $x \mapsto 2x-2^n$). Given the endpoints, we couple the walks so that $S'_i \ge S_i$ for all $i$ when $S'_{2^n} \ge S_{2^n}$, and arbitrarily otherwise. The former can be done by first sampling $(S_i)$ and then uniformly choosing $\frac12(S'_{2^n}-S_{2^n})$ coordinates $i$ among those where the increment $S_i-S_{i+1}$ is $-1$ and setting the corresponding increments $S'_i-S'_{i-1}$ to $+1$ there (with all other increments remaining the same for both). This completes the description of the coupling between $(X_t)$ and $(X'_t)$.

It remains to check that the constructed coupling has the claimed property. Let $\Delta_n = X_{t_{n+1}}-X_{t_n}$ and $\Delta'_n = X'_{t_{n+1}}-X'_{t_n}$. Define events
\[ E_n = \{ \Delta'_n - \Delta_n \ge 2^{n/2} / n^2 \} \qquad\text{and}\qquad F_n = \{ \Delta'_n<\Delta_n \} .\]
Since $F_n$ has probability at most $C'/n^2$, only finitely many of the $F_n$ occur almost surely. Let $N_1$ be the smallest positive integer such that $F_n$ does not occur for any $n \ge N_1$. Observe that $X'_t-X_t$ is non-decreasing for $t \ge t_{N_1}$.
Since $\{E_n\}_{n=1}^\infty$ are independent events, each of probability at least $c'$, infinitely many of them occur almost surely. Moreover, almost surely, for any $n$ large enough, at least one of $E_{n-1},E_{n-2},\dots,E_{n-C''\log n}$ occurs, where $C''>0$ is some large constant. Let $N_2$ be the smallest positive integer so that this holds for $n \ge N_2$. Observe that if $n-C''\log n \ge \max\{N_1,N_2\}$ and $t_n \le t \le t_{n+1}$, then letting $n-C''\log n \le m<n$ be such that $E_m$ occurs, we obtain that
\begin{align*}
 X'_t-X_t
  &\ge X'_{t_{m+1}} - X_{t_{m+1}} \\
  &= \Delta'_m-\Delta_m  + X'_{t_m} - X_{t_m} \\
  &\ge 2^{m/2}/m^2 + X'_{t_N}-X_{t_N} \\
  &\ge \sqrt t \cdot e^{-5C\log\log t} ,
\end{align*}
where the last inequality holds for $t$ large enough.
\end{proof}

\begin{proof}[Proof of \cref{thm:lazy-bad2}]
Let $(X'_t)$ denote a copy of the lazy random walk $(X_t)$.
By the first lemma, $(X'_t)$ and $(Y_t)$ can be coupled so that $|X'_t|-|Y_t| \le 2\log t$ eventually. By the second lemma (and the first remark following it), 
the walks $(|X_t|)$ and $(|X'_t|)$ can be coupled so that $|X'_t|-|X_t| \ge \sqrt t / \log^C t$ eventually. Extend this coupling to a coupling of $(X_t)$ and $(X'_t)$. The processes $(X_t)$ and $(Y_t)$ are now coupled so that $|Y_t|-|X_t| \ge \sqrt t / \log^C t -2\log t \ge \sqrt t / \log^{2C} t$ eventually.
\end{proof}

Removing the automorphism assumption in \cref{thm:lazy-bad22}, even at the expense of increasing the upper bound on $|X_t|-|Y_t|$ from $2\log t$ to $\sqrt t / (\log t)^C$ for a sufficiently large constant $C$, would suffice in order to lift the automorphism assumption in \cref{thm:lazy-bad2}.

\subsection*{Acknowledgements}

OA would like to thank the American Institute of Math, where this project was initiated, and the Technion, where the collaboration began. OA and YS are supported in part by NSERC. 
AY is partially supported by the BSF.

\bibliographystyle{abbrv}
\bibliography{permuted-RW-20}

\end{document}